%% file: m.tex
\documentclass[reqno,12pt]{amsart}

\include{environ}
\include{symbols}
\usepackage[all,tips]{xy}

\begin{document}

\title[Adaptive Mixed Methods on Surfaces]
      {Convergence and Optimality of \\
       Adaptive Mixed Methods on Surfaces}     
  
\author[M. Holst, A. Mihalik, R. Szypowski ]{Michael Holst, Adam Mihalik, and Ryan Szypowski}  

\address{Department of Mathematics\\
         University of California San Diego\\ 
         La Jolla CA 92093}
\email{mholst@math.ucsd.edu, amihalik@math.ucsd.edu}

\address{Department of Mathematics and Statistics\\
         Cal Poly Pomona\\ 
         Pomona CA 91768}
\email{rsszypowski@csupomona.edu}

\thanks{MH was supported in part by NSF Awards~1217175, 1065972, and 1262982.}
\thanks{AM was supported in part by NSF Award~1217175.}
\thanks{RS was supported in part by NSF Award~1216868.}

\date{\today}


\input{abs}

\maketitle


\vspace*{-1.0cm}
{\footnotesize
\tableofcontents
}
\vspace*{-1.0cm}


\input{body}

\bibliographystyle{abbrv}
\bibliography{zbib/adam,zbib/mjh-papers,zbib/mjh-extra,zbib/HoSt2010b,zbib/library}


\vspace*{0.5cm}

\end{document}

%% file: environ.tex
\usepackage{color} 
\usepackage{ifpdf}
\ifpdf
    \usepackage[pdftex]{graphicx}
    \usepackage[pdftex]{hyperref}
    \hypersetup{
        unicode=false,          
        pdftoolbar=true,        
        pdfmenubar=true,        
        pdffitwindow=false,     
        pdfstartview={FitH},    
        pdftitle={MCP Article},      
        pdfauthor={Michael Holst},   
        pdfsubject={Mathematics},    
        pdfcreator={Michael Holst},  
        pdfproducer={Michael Holst}, 
        pdfkeywords={PDE, analysis, mathematical physics}, 
        pdfnewwindow=true,      
        colorlinks=true,        
        linkcolor=red,          
        citecolor=blue,         
        filecolor=magenta,      
        urlcolor=cyan           
    }

\else
    \usepackage{graphicx}
    \usepackage{pstricks}
    
    \newcommand{\href}[2]{#2}
\fi

\usepackage{times}
\usepackage{amsfonts}

\usepackage{amsmath}
\usepackage{amsthm}
\usepackage{amssymb}
\usepackage{amsbsy}
\usepackage{amscd}

\usepackage{enumerate}
\usepackage{verbatim}
\usepackage{subfigure}

\newtheorem{theorem}{Theorem}[section]
\newtheorem{corollary}[theorem]{Corollary}
\newtheorem{lemma}[theorem]{Lemma}

\newtheorem{definition}[theorem]{Definition}

\numberwithin{equation}{section}  

\newcounter{mnote}
\makeatletter
\newsavebox{\@brx}
\newcommand{\llangle}[1][]{\savebox{\@brx}{\(\m@th{#1\langle}\)}%
  \mathopen{\copy\@brx\kern-0.5\wd\@brx\usebox{\@brx}}}
\newcommand{\rrangle}[1][]{\savebox{\@brx}{\(\m@th{#1\rangle}\)}%
  \mathclose{\copy\@brx\kern-0.5\wd\@brx\usebox{\@brx}}}
\makeatother

\definecolor{myblue}{rgb}{0.2,0.2,0.7}
\definecolor{mygreen}{rgb}{0,0.6,0}
\definecolor{mycyan}{rgb}{0,0.6,0.6}
\definecolor{myred}{rgb}{0.9,0.2,0.2}
\definecolor{mymagenta}{rgb}{0.9,0.2,0.9}
\definecolor{mywhite}{rgb}{1.0,1.0,1.0}
\definecolor{myblack}{rgb}{0.0,0.0,0.0}

%% file: symbols.tex
\newcommand{\beq}{\begin{equation}}
\newcommand{\eeq}{\end{equation}}
\newcommand{\beqa}{\begin{eqnarray}}
\newcommand{\eeqa}{\end{eqnarray}}

%% file: abs.tex
\begin{abstract}
In a 1988 article, Dziuk introduced a nodal finite element method for the Laplace-Beltrami equation on 2-surfaces approximated by a piecewise-linear triangulation, initiating a line of research into surface finite element methods (SFEM).  Demlow and Dziuk built on the original results, introducing an adaptive method for problems on 2-surfaces, and Demlow later extended the \textit{a priori} theory to 3-surfaces and higher order elements.  In a separate line of research, the Finite Element Exterior Calculus (FEEC) framework has been developed over the last decade by Arnold, Falk and Winther and others as a way to exploit the observation that mixed variational problems can be posed on a Hilbert complex, and Galerkin-type mixed methods can be obtained by solving finite dimensional subproblems.  In 2011, Holst and Stern merged these two lines of research by developing a framework for variational crimes in abstract Hilbert complexes, allowing for application of the FEEC framework to problems that violate the subcomplex assumption of Arnold, Falk and Winther.  When applied to Euclidean hypersurfaces, this new framework recovers the original \textit{a priori} results and extends the theory to problems posed on surfaces of arbitrary dimensions.   In yet another seemingly distinct line of research, Holst, Mihalik and Szypowski developed a convergence theory for a specific class of adaptive problems in the FEEC framework.  Here, we bring these ideas together, showing convergence and optimality of an adaptive finite element method for the mixed formulation of the Hodge Laplacian on hypersurfaces.
\end{abstract}

%% file: body.tex

\section{Introduction}

Adaptive finite element methods (AFEM)  based on \textit{a posteriori} error estimators have become standard tools in solving PDE problems arising in science and engineering (cf.~\cite{Ainsworth.M;Oden.J2000,Verfurth.R1996,Repin.S2008}).  A fundamental difficulty with these adaptive methods 
is guaranteeing convergence of the solution sequence.
The first convergence result was obtained by 
Babuska and Vogelius~\cite{Babuska.I;Vogelius.M1984} for 
linear elliptic problems in one space dimension, and many improvements and generalizations to the theory have followed ~\cite{Dorfler.W1996, MS1,Morin.P;Siebert.K;Veeser.A2007,Morin.P;Siebert.K;Veeser.A2008,Siebert.K2009}.  
Convergence, however, does not necessarily imply optimality of a method.  This idea has led to the development of a theory related to the optimal computational complexity of AFEM, and within this framework certain classes of adaptive methods have been shown to be optimal  ~\cite{BBD04,RS07,CKNS}.  

In a 1988 article, Dziuk \cite{DZ88} introduced a nodal finite element method for the Laplace-Beltrami equation on 2-surfaces approximated by a piecewise-linear triangulation, pioneering a line of research into \textit{surface finite element (SFEM)} methods.  Demlow and Dziuk \cite{DD07} built on the original results, introducing an adaptive method for problems on 2-surfaces, and Demlow later extended the \textit{a priori} theory to 3-surfaces and higher order elements \cite{DEM09}.  While \textit{a posteriori} error indicators are introduced and shown to have desirable properties in \cite{DD07}, a convergence and optimality theory related to problems on surfaces is a relatively undeveloped area, and developing such a theory is the main topic of this article.

A separate idea that has had a major influence on the
development of numerical methods for PDE applications
is that of \textit{mixed finite elements}, whose early success in areas such as computational
electromagnetics was later found to have surprising connections with the calculus of exterior differential forms, including de~Rham
cohomology and Hodge theory~\cite{Bossavit1988,Nedelec1980,Nedelec1986,GrKo2004}.  Around the same time period, Hilbert complexes were studied as a way to generalize certain properties of elliptic complexes, particularly the Hodge decomposition and Hodge theory \cite{BrLe1992}.  These ideas came together with the introduction of the theory of  \textit{finite element exterior calculus} (FEEC), where Arnold, Falk and Winther showed that Hilbert complexes were a natural setting for analysis and numerical approximation of mixed variational problems  by mixed finite elements.  This theory has proved a powerful tool in developing general results related to mixed finite elements.  In \cite{HoSt10a,HoSt10b}, Holst and Stern extend the theory to include problems in which the discrete complex is not a subcomplex of the approximated complex, and applying these results they develop an \textit{a priori} theory for the Hodge Laplacian on hypersurfaces, and to nonlinear problems.  This made it possible in~\cite{HoSt10a} to reproduce the existing \textit{a priori} theory for SFEM as a particular application, as well as to generalize SFEM theory in several directions.  In \cite{HMS}, we used the FEEC framework as a critical tool for developing an AFEM convergence theory for a class of adaptive methods for linear problems posed on domains in $\mathbb{R}^n$.  The aim of this paper is to build upon these results and develop a convergence theory for a class of problems that violate the subcomplex assumption of Arnold, Falk and Winther, allowing for the treatment of problems on surfaces.

More specifically, we introduce an adaptive method for problems posed on smooth Euclidean hypersurfaces in which finite element spaces are mapped from a fixed approximating polygonal manifold.  The mesh on the fixed approximating surface will be refined using error indicators related to the original problem.  Using tools developed in \cite{HoSt10a,HoSt10b}, the auxiliary results of \cite{HMS} are modified to account for the surface mapping, yielding an adaptive method whose main results mirror those of \cite{HMS}.  In doing this we  establish the optimality of a convergent algorithm for the Hodge Laplacian (case $k=m$) on hypersurfaces of arbitrary dimension.

The remainder of the paper is organized as follows. In Section~\ref{sec:prelim} we introduce the notational and technical tools essential for the paper.  We begin by discussing the fundamental framework of abstract Hilbert complexes and in particular the de Rham complex \cite{AFW10}, ideas which are critical in the development of the theory of finite element exterior calculus.  We then finish the section with a brief overview  of a standard adaptive finite element algorithm.  Next, Section ~\ref{sec:GeoPrelim} follows \cite{HoSt10a,HoSt10b} by introducing geometric tools and ideas that tie the general theory developed in \cite{AFW06, AFW10} to problems on Euclidean hypersurfaces.  Additionally we prove some basic results for an interpolant built on the approximating surface.  In Section~\ref{sec:quasi} we closely follow the ideas in~\cite{HMS}
and develop a similar quasi-orthogonality result, specifically tailoring our results for application on surfaces.
Section~\ref{sec:stab} again closely follows \cite{HMS}, and we prove a discrete stability result applicable to problems on surfaces (which is needed
for proving quasi-orthogonality in Section~\ref{sec:quasi}), and also
establish a continuous stability result, which will be needed for deriving
an upper bound on the error.
In Section~\ref{sec:bounds} we begin by introducing an error indicator 
and then derive bounds and a type of continuity result for this indicator.
An adaptive algorithm is then presented in Section~\ref{sec:conv}, for which
 convergence and optimality are proved using the auxiliary results from the previous sections.
Finally, we close in Section~\ref{sec:sConclusion} with a discussion on related future directions and alternative methods for solving numerical PDE on surfaces.   The results in this paper follow \cite{HMS} in a natural manner.  It is the same convergence idea, but the results are adapted to account for the geometry of the surface and the mapping between the surfaces.
 
\section{Notation and Framework}
\label{sec:prelim}

The algorithm developed in this article will rely heavily on the methods introduced on polygonal domains in \cite{HMS}.  In order to keep this work self contained, this section will provide a similar introduction to that of \cite{HMS}, from which we quote freely.  We begin with an introduction of some basic concepts of abstract Hilbert complexes.
Next, we examine the particular case of the de Rham complex, closely following the notation and general development 
of Arnold, Falk and Winther in ~\cite{AFW06,AFW10}.
We also discuss results from Demlow and Hirani in~\cite{DH}.
(See also~\cite{HoSt10a,HoSt10b} for a concise summary of Hilbert complexes
in a yet more general setting.)
We then give an overview of the basics of adaptive finite element methods 
(AFEM), and the ingredients we will need to prove convergence and optimality
within the FEEC framework.
 
\subsection{Hilbert Complexes}

A \textit{Hilbert complex} $(W, d )$ is a sequence of Hilbert spaces $W^k$ equipped with closed, densely defined linear operators, $d^k$, which map their domain, $V^k \subset W^k$ to the kernel of $d^{k+1}$ in $W^{k+1}$.  A Hilbert complex is \textit{bounded} if each $d^k$ is a bounded linear map from $W^k$ to $W^{k+1}$  A Hilbert complex is \textit{closed} if the range of each $d^k$ is closed in $W^{k+1}$.
Given a Hilbert complex $(W, d)$, the subspaces  $V^k \subset W^k$ endowed with the graph inner product
\begin{equation*}
\langle u, v \rangle_{V^k} = \langle u, v \rangle_{W^k} + \langle d^k u, d^k v \rangle_{W^{k+1} },
\end{equation*}
form a Hilbert complex $(V, d)$ known as the \textit{domain complex}.
By definition $d^{k+1}\circ$  $d^k = 0$, thus  $(V, d)$ is a bounded 
Hilbert complex.
Additionally, $(V, d)$ is closed if $(W, d)$ is closed.

The range of $d^{k-1}$ in $V^k$ will be represented by $\mathfrak{B}^k$, and the null space of $d^k$ will be represented by $\mathfrak{Z}^k.$  Clearly, $\mathfrak{B}^k \subset \mathfrak{Z}^k$.  The elements of $\mathfrak{Z}^k$ orthogonal to $\mathfrak{B}^k$ are the space of harmonic forms, represented by $\mathfrak{H}^k$.  For a closed Hilbert complex we can write the \textit{Hodge decomposition} of $W^k$ and $V^k$,
\begin{align}
W^k &= \mathfrak{B}^k \oplus \mathfrak{H}^k \oplus \mathfrak{Z}^{k \perp_W }, \\
V^k &= \mathfrak{B}^k \oplus \mathfrak{H}^k \oplus \mathfrak{Z}^{k \perp_V }.
\end{align}
Following notation common in the literature, we will write $\mathfrak{Z}^{k \perp}$ for $\mathfrak{Z}^{k \perp_W }$ or $\mathfrak{Z}^{k \perp_V }$, when clear from the context.
Another important Hilbert complex will be the \textit{dual complex} $(W, d^*)$, where $d^*_k$, which is an operator from $W^k$ to $W^{k-1}$, is the adjoint of $d^{k-1}$.  The domain of $d^*_k$ will be denoted by $V^*_k$.
For closed Hilbert complexes, an important result will be the \textit{Poincar\'e inequality},
\begin{equation}\label{poincare}
\|v\|_V \le c_P\|d^kv\|_W, \hspace{.2cm} v\in \mathfrak{Z}^{k\perp}.
\end{equation}
The de Rham complex is the practical complex where general results we show on an abstract Hilbert complex will be applied.  The de Rham complex satisfies an important compactness property discussed in \cite{AFW10}, and therefore this compactness property is assumed in the abstract analysis. 

\subsubsection*{The Abstract Hodge Laplacian}

Given a Hilbert complex $(W,d)$, the operator $L = dd^* + d^*d$,  $W^k \rightarrow W^{k}$ will be referred to as the \textit{abstract Hodge Laplacian}.  For $f\in W^k$, the Hodge Laplacian problem can be formulated weakly as the problem of finding $u \in W^k$ such that
\begin{equation*}
\langle du, dv \rangle +\langle d^* u, d^* v \rangle = \langle f, v \rangle, v \in V^k \cap V^*_k.
\end{equation*}

The above formulation has undesirable properties from a computation perspective.  The finite element spaces $V^k \cap V^*_k$ can be difficult to implement, and the problem will not be well-posed in the presence of a non-trivial harmonic space, $\mathfrak{H}^k$. In order to circumvent these issues, a well posed (cf.~\cite{AFW06, AFW10}) \textit{mixed formulation of the abstract Hodge Laplacian} is introduced as the problem of finding $( \sigma, u, p ) \in V^{k-1} \times V^k \times \mathfrak{H}^k$, such that:
\begin{equation}\label{HL}
\begin{array}{rll} 
\langle \sigma, \tau \rangle - \langle d\tau, u \rangle& =  0, & \forall \tau \in V^{k-1},  \\ 
\langle d \sigma, v \rangle + \langle du,  dv \rangle+\langle p, v  \rangle& =  \langle f, v\rangle, & \forall v \in V^k ,\\ 
\langle u, q \rangle  &= 0, & \forall q \in \mathfrak{H}^k.
\end{array}
\end{equation} 

\subsubsection*{Sub-Complexes and Approximate Solutions to the Hodge Laplacian }

In ~\cite{AFW06,AFW10} a theory of approximate solutions to the Hodge-Laplace problem is developed by using finite dimensional approximating Hilbert complexes.  Let $(W,d)$ be a Hilbert complex with domain complex $(V,d)$.  An approximating subcomplex is a set of finite dimensional Hilbert spaces, $V_h^k \subset V^k$ with the property that $dV^{k}_h \subset V_h^{k+1}$.  Since $V_h$ is a Hilbert complex, $V_h$ has a corresponding Hodge decomposition,
\begin{equation*}
V_h^k = \mathfrak{B}_h^k \oplus \mathfrak{H}_h^k \oplus \mathfrak{Z}_h^{k \perp_V }.
\end{equation*}
By this construction, $(V_h,d)$ is an abstract Hilbert complex with a well posed Hodge Laplace problem.  Find $( \sigma_h, u_h, p_h ) \in V^{k-1}_h \times V^k_h \times \mathfrak{H}^k_h$, such that
\begin{equation}\label{DHL}
\begin{array}{rll} 
\langle \sigma_h, \tau \rangle - \langle d\tau, u_h \rangle& =  0, & \forall \tau \in V^{k-1}_h,  \\ 
\langle d \sigma_h, v \rangle + \langle du_h,  dv \rangle+\langle p_h, v  \rangle& =  \langle f, v\rangle, & \forall v \in V^k _h,\\ 
\langle u_h, q \rangle  &= 0, & \forall q \in \mathfrak{H}^k_h .
\end{array}
\end{equation} 
An assumption made in~\cite{AFW10} in developing this theory is the existence of a bounded cochain projection, $\pi_h:V \rightarrow V_h$, which commutes with the differential operator.

In~\cite{AFW10}, an \textit{a priori} convergence result is developed for the solutions on the approximating complexes.  The result relies on the approximating complex getting sufficiently close to the original complex in the sense that inf$_{v \in V_h^k} \|u-v\|_V$ can be assumed sufficiently small for relevant $u \in V^k$.  Adaptive methods, on the other hand, gain computational efficiency by limiting the degrees of freedom used in areas of the domain where it does not significantly impact the quality of the solution.

\subsection{The de Rham Complex and its Approximation Properties}
The de Rham complex is a cochain complex where the abstract results from the previous section can be applied in developing practical computational methods.  This section reviews concepts and definitions related to the de Rham complex necessary in our development of an adaptive finite element method.  This introduction will be brief and and mostly follows the notation from the more in-depth discussion in \cite{AFW10}.  

In order to introduce the ideas of \cite{HMS}, we first assume a 
bounded Lipschitz polyhedral domain, $\Omega \in \mathbb{R}^n, n\ge 2$.
Let $\Lambda^k(\Omega)$ be the space of smooth $k$-forms on $\Omega$, 
and let $L^2\Lambda^k(\Omega)$ be the completion of $\Lambda^k(\Omega)$ with respect to the $L^2$ inner-product.
There are no non-zero harmonic forms in $L^2 \Lambda^n(\Omega)$
(see~\cite{AFW06}, Theorem~2.4) which will often simplify the 
analysis in our primary case of interest, $k = n$.
For general $k$ such a property cannot be assumed, and therefore, since the $\mathfrak{B}$ problem deals with the spaces of $k$ and $({k-1})$-forms, analysis of the harmonic spaces is still necessary.
Note that the results in~\cite{CHX} hold only for polygonal and simply connected domains, therefore  $\mathfrak{H}^{k-1}$ is also void in the case $k=n=2$.

\subsubsection*{The de Rham Complex}
  
Let $d$ be the exterior derivative acting as an operator from $L^2\Lambda^k(\Omega)$ to $L^2\Lambda^{k+1}(\Omega)$.  The $L^2$ inner-product will define the $W$-norm, and the $V$-norm will be defined as the graph inner-product 
\begin{equation*}
\langle u, \omega \rangle_{V^k} = \langle u, \omega \rangle_{L^2} + \langle du, d\omega \rangle_{L^2}.
\end{equation*}
This forms a Hilbert complex $(L^2\Lambda(\Omega), d)$, with domain complex $(H\Lambda(\Omega),d)$, where $H\Lambda^k(\Omega)$ is the set of elements in $L^2\Lambda^k(\Omega)$ with exterior derivatives in $L^2\Lambda^{k+1}(\Omega)$. The domain complex can be described with the following diagram
\begin{equation}
0\rightarrow H\Lambda^0( \Omega) \xrightarrow{d} \cdots  \rightarrow    H\Lambda^{n-1}( \Omega ) \xrightarrow{d} L^2(\Omega) \xrightarrow{} 0.
\end{equation}
It can be shown that the compactness property is satisfied, and therefore the prior results shown on abstract Hilbert complexes can be applied. 
  
The Hodge star operator,  $\star: \Lambda^k(\Omega) \rightarrow \Lambda^{n-k}(\Omega)$, is then defined using the wedge product.  For $\omega \in \Lambda^k(\Omega)$, 
\begin{equation*}
\int_\Omega \omega \wedge \mu = \langle  \star \omega, \mu \rangle_{L^2 \Lambda^{n-k}}, \hspace{.2cm} \forall \mu \in \Lambda^{n-k}(\Omega).
\end{equation*}
Next we introduce the coderivative operator, $\delta: \Lambda^k(\Omega) \rightarrow \Lambda^{k-1}(\Omega)$,  
\begin{equation}
\star \delta \omega = (-1)^k d \star \omega,\label{delta}.
\end{equation}
which combined with Stokes theorem allow integration by parts to be written as
\begin{equation}
\langle d \omega, \mu \rangle = \langle \omega, \delta \mu \rangle + \int_{\partial\Omega} \text{tr } \omega \wedge \text{tr} \star \mu, \hspace{.2cm}
 \omega \in \Lambda^{k-1}, \text{ }\mu \in \Lambda^k(\Omega).
\end{equation}
Using this formulation and the following spaces,
\begin{align*}
\mathring{H}\Lambda^k(\Omega) &= \{ \omega \in H\Lambda^k(\Omega) \vline \text{  tr}_{\partial\Omega} \omega = 0 \}, \\
\mathring{H}^{*}\Lambda^k(\Omega) &:= \star \mathring{H}\Lambda^{n-k}(\Omega),
\end{align*}
the following theorem connects the framework built for abstract Hilbert complexes to the de Rham complex.
\begin{theorem}\label{afw41}(Theorem 4.1 from~\cite{AFW10})
Let d be the exterior derivative viewed as an unbounded operator $L^2\Lambda^{k-1}(\Omega) \rightarrow L^2\Lambda^k(\Omega)$ with domain $H\Lambda^k(\Omega)$.  The the adjoint d$^*$, as an unbounded operator $L^2\Lambda^k(\Omega) \rightarrow L^2\Lambda^{k-1}(\Omega)$, has $\mathring{H}^{*}\Lambda^k(\Omega) $ as its domain and coincides with the operator $\delta$ defined in \eqref{delta}.
\end{theorem}

Applying the results from the previous section and Theorem \ref{afw41}, we get the mixed Hodge Laplace problem on the de Rham complex: find the unique $(\sigma,u,p) \in H\Lambda^{k-1}(\Omega) \times H\Lambda^k(\Omega) \times \mathfrak{H}^k$ such that
\begin{equation}\label{GHDR}
\begin{array}{cl} 
\sigma = \delta u, \text{ } d\sigma + \delta d u = f - p  &\text{ in } \Omega, \\
\text{tr} \star u = 0, \text{ tr} \star du = 0 &\text{ on } \partial\Omega,  \\
u \perp \mathfrak{H}^k. &
\end{array}
\end{equation}


\subsubsection*{Finite Element Differential Forms}
For the remainder of the paper it is assumed that all approximating sub-complexes of the de Rham complex are constructed as combinations of the polynomial spaces of $k$-forms, $\mathcal{P}_r\Lambda^k$ and $\mathcal{P}^{-}_r\Lambda^k$.  For a detailed discussion on these spaces and construction of Hilbert complexes using these spaces, see \cite{AFW10}.  We also have useful properties in the case $k = n$,
\begin{align}
 \mathcal{P}^{-}_r\Lambda^n &= \mathcal{P}_{r-1}\Lambda^n, \\
 \mathcal{P}^{-}_r\Lambda^0 &= \mathcal{P}_{r}\Lambda^0.
\end{align}

For a shape-regular, conforming triangulation $\mathcal{T}_h$ of $\Omega, \Lambda_h^k(\Omega) \subset L^2\Lambda^k(\Omega)$ will denote a space of $k$-forms constructed using specific combinations of the these spaces on $\mathcal{T}_h$.  For an element $T \in \mathcal{T}_h$, we set $h_T :=$ diam$(T)$.
We do not discuss the details of these spaces further, but specific properties will be explained when necessary.

\subsubsection*{Bounded Cochain Projections}
Bounded cochain projections and their approximation properties are necessary in the analysis of both uniform and adaptive FEMs in the FEEC framework.  Properties of three different interpolation operators will be important in our analysis.  The three operators and respective notation that we will use are as follows: the canonical projections $I_h$ defined in \cite{AFW06, AFW10}, the smoothed projection operator $\pi_h$ from \cite{AFW10}, and the commuting quasi-interpolant $\Pi_h$, as defined in \cite{DH} with ideas similar to \cite{sch01,sch08,cw08}.  Some cases will require a simple projection, and $P_h f$ also written  $f_h$, will denote the $L^2$-projection of $f$ on to the discrete space parameterized by $h$.  

For the remainder of the paper,  $\| \cdot \|$ will denote the $L^2\Lambda^k(\Omega)$ norm, and when taken on specific elements of the domain, $T$, we write $\|\cdot\|_T$.  For all other norms, such as $H\Lambda^k(\Omega)$ and $H^1\Lambda^k(\Omega)$, we write $\|\cdot\|_{H\Lambda^k(\Omega)}$ and $\|\cdot\|_{H^1\Lambda^k(\Omega)}$ respectively.

\begin{lemma}\label{canProj}  Suppose $\tau \in H^1\Lambda^{k}(\Omega)$, where $k = n-1$ or $k = n $.  Let $I_h$ be the canonical projection operator defined in \cite{AFW06, AFW10} and let $\Lambda^{n-1}_h(\Omega)$ and $\Lambda^n_h(\Omega)$ be defined as above. Then $I_h$ is a projection onto $\Lambda^n_h(\Omega), \Lambda^{n-1}_h(\Omega) $ and satisfies
 \begin{align}
 \| \tau - I_h \tau \|_{T} &\le Ch_T \|\tau \|_{H^1\Lambda^{k}(T) }, \hspace{ .3cm } \forall T \in \mathcal{T}_h, \label{h1interp}\\
 I_h d &= d I_h  
 \end{align}
\end{lemma} 
\begin{proof}
The first part is comes from Equation (5.4) in \cite{AFW06}.  The second part follows the construction of $I_h$.
\end{proof} 

Lemmas \ref{PL1} and \ref{PL2} deal with important properties of the canonical projections.  In each case we assume $f _h, u_h \in \Lambda^n_h( \Omega)$, and let $\mathcal{T}_h$ be a refinement of $\mathcal{T}_H$.

\begin{lemma} \label{PL1}   Let $T \in \mathcal{T}_H$, then

 \begin{equation} \int_T (f_h - I_H f_h ) = 0. \end{equation}\label{discInt}
\end{lemma}
 
\begin{proof}
See \cite{HMS}
\end{proof}

\begin{lemma}\label{PL2}  Let $T \in \mathcal{T}_H$, then
  \begin{equation}
 \langle( I_h - I_H ) u_h , f_h \rangle_T = \langle u_h , ( I_h - I_H )  f_h \rangle_T.
\end{equation}
\end{lemma}

\begin{proof}
See \cite{HMS}.
 \end{proof}

$\linebreak$
The next lemma is taken directly from \cite{DH}, and will be a key tool in developing an upper bound for the error.

\begin{lemma}\label{dhi}

Assume 1 $\le k \le $ n, and $\phi \in H\Lambda^{k-1}(\Omega)$ with $\| \phi \| \le 1.$ 
 Then there exists $\varphi \in H^1 \Lambda^{k-1}(\Omega)$ such that
 $d\varphi = d\phi, \Pi_H d \phi = d\Pi_H \phi = d\Pi_H \varphi$, and

\begin{equation*}
 \displaystyle\sum\limits_{T \in T_h}  h_T^{-2} \| \varphi - \Pi_H\varphi \|_{T} ^2 +   h_T^{-1} \| \textit{tr} ( \varphi - \Pi_H \varphi ) \|_{\partial T }^2 \le C.
\end{equation*}
\end{lemma}

\begin{proof} See Lemma 6 in \cite{DH}.
\end{proof}

The following theorem is a special case of Theorem 3.5 from \cite{AFW10}.  Rather than showing the result on an abstract Hilbert Complex with a general cochain projection, we use the de Rham complex and the smoothed projection operator $\pi_h$ in order to use uniform boundedness of the cochain projection.
\begin{theorem} Assume $\Lambda_h^k(\Omega)$ is a subcomplex of $H\Lambda^k(\Omega)$ as described above, and let $\pi_h$ be the smoothed projection operator.  Then
\begin{equation}\label{hIntAFW}  
\|(I-P_{\mathfrak{H}^k})q \|_V \le \| (I - \pi_h^k)P_{\mathfrak{H}^k}q \|_V, \hspace{.5cm} q\in \mathfrak{H}_h^k,
\end{equation}
\end{theorem}
then combining the above with the triangle inequality,
\begin{equation}\label{AFW30}
\|q\|_V \le c\|P_{\mathfrak{H}^k }q \|_V, \hspace{.5cm} q\in \mathfrak{H}_h^k.
\end{equation}
\begin{proof}  See \cite{HMS}.
\end{proof}

Theorem \ref{chdt} will be essential in dealing with the harmonic forms in the proof of a continuous upper-bound.  The corollary will be used identically when proving a discrete upper-bound.  For use in our next two results we introduce an operator $\delta$ and one of its important properties.  Let $A, B$ be $n < \infty$ dimensional, closed subspaces of a Hilbert space $W$, and let
\begin{equation*}
\delta( A, B ) = \sup_{x\in A, \|x\| = 1} \|x - P_Bx \|,
\end{equation*}
then \cite{DH}, Lemma 2 which takes the original ideas from \cite{KATO}, shows
\begin{equation}\label{DHL2}
\delta(A,B) = \delta(B, A ).
\end{equation}

\begin{theorem}\label{chdt}  Assume  $\mathfrak{H}_H^k$ and $\mathfrak{H}^k$ have the same finite dimensionality. The there exist a constant $C_{\mathfrak{H}^k}$  dependent only on $\mathcal{T}_0$, such that 
\begin{equation} \label{chd} 
 \delta( \mathfrak{H}^k, \mathfrak{H}^k_H ) = \delta( \mathfrak{H}^k_H, \mathfrak{H}^k ) \le C_{\mathfrak{H}^k} < 1 .
\end{equation}
\end{theorem}
\begin{proof}
See \cite{HMS}
\end{proof}

\begin{corollary}\label{dhdc}
\begin{equation} \label{dhd} 
 \delta( \mathfrak{H}_h^k, \mathfrak{H}_H^k ) = \delta( \mathfrak{H}_H^k, \mathfrak{H}_h^k ) \le \tilde{C}_{\mathfrak{H}^k}  < 1 .
\end{equation}
\begin{proof}
The proof follows the same logic as Theorem \ref{chdt}.  The only difference is that the harmonics are compared on two discrete complexes $\mathfrak{H}^k_h$ and $\mathfrak{H}^k_H$, and therefore $I_h$ is used rather than $\pi_h$.
\end{proof}
\end{corollary}

\subsection{Adaptive Finite Elements Methods}
This section gives a concise introduction to key concepts and notation used in developing our AFEM.  Our methods will follow \cite{ RS07, MS1, stars, DORFLER96, CHX}, which give more a more complete discussion on AFEM.

Given an initial triangulation, $\mathcal{T}_0$, the adaptive procedure will generate a nested sequence of triangulations $\mathcal{T}_k$ and discrete solutions $\sigma_k$, by looping through the following steps:
\begin{equation}
 \textsf{Solve} \longrightarrow \textsf{Estimate} \longrightarrow \textsf{Mark} \longrightarrow \textsf{Refine}
 \end{equation}
The following subsection will describe details of these steps.
\subsubsection*{Approximation Procedure}
We assume access to a routine \textsf{SOLVE}, which can produce solution to \eqref{DHL} given a triangulation, problem data, and a desired level of accuracy.   For the \textsf{ESTIMATE} step we will introduce error indicators $\eta_T$ on each element $T \in \mathcal{T}_k$.  In the \textsf{MARK} step we will use D{\"o}rfler Marking strategy \cite{DORFLER96}.  An essential feature of the marking process is that the summation of the error indicators on the marked elements exceeds a user defined marking parameter $\theta$.

We assume access to an algorithm \textsf{REFINE} in which marked elements are subdivided into two elements of the same size, resulting in a conforming, shape-regular mesh.  Triangles outside of the original marked set may be refined in order to maintain conformity.  Bounding the number of such refinements is important in showing optimality of the method.  Along these lines, Stevenson \cite{RS} showed certain bisection algorithms developed in two-dimensions can be extended to $n$-simplices of arbitrary dimension satisfying  
\begin{align*}
&(1) \{\mathcal{T}_k\} \text{ is shape regular and the shape regularity depends only on } \mathcal{T}_0, \\
&(2) \# \mathcal{T}_k \le \#\mathcal{T}_0 + C\#\mathcal{M},
\end{align*}
where $\mathcal{M}$ is the collection of all marked triangles going from $\mathcal{T}_0$ to $\mathcal{T}_k$. 

\subsubsection*{Approximation of the Data}  
A measure of data approximation will be necessary in establishing a quasi-orthogonality result.  Following ideas of \cite{MS1}, data oscillation will be defined as follows,
\begin{definition}(\textit{Data oscillation})  Let f $\in L^2\Lambda^k(\Omega)$, and $\mathcal{T}_h$ be a conforming triangulation of $\Omega$.  Let $h_T$ be the diameter for a given $T \in \mathcal{T}_h$.  We define
\begin{equation*}
\textit{osc}(f, \mathcal{T}_h ) := \big( \displaystyle\sum\limits_{T \in \mathcal{T}_h}  \|h_T(f-f_{h})\|^2_T \big)^{1/2}.
\end{equation*}
\end{definition}

Stevenson \cite{RS} generalized the ideas of \cite{BBD04} to show that approximation of data can be done in an optimal way regardless of dimension.  Using the approximation spaces $(\mathcal{A}^s, \| \cdot \|_{\mathcal{A}^s} )$ and $(\mathcal{A}^s_o, \| \cdot \|_{\mathcal{A}^s_o} )$ as in \cite{BBD04} we recall the result.

\begin{theorem}\label{CHX22}\textit{(Generalized Binev, Dahmen and DeVore) }
Given a tolerance $\epsilon$,  f $\in L^2\Lambda^n(\Omega)$ and a shape regular triangulation $\mathcal{T}_0$, there exists an algorithm
\begin{center} $\mathcal{T}_H = \textit{\textsf{APPROX}}( f, \mathcal{T}_0, \epsilon )$,  \end{center}
such that
\begin{equation*}
\textit{osc}( f, \mathcal{T}_H) \le \epsilon,\hspace{.3cm} and \hspace{.3cm} \#\mathcal{T}_H - \#\mathcal{T}_0 \le C \|f\|^{1/s}_{\mathcal{A}_o^{1/s}} \epsilon^{-1/s}.
\end{equation*}
\end{theorem}

As in the case of \cite{CHX}, the analysis of convergence and procedure will follow \cite{CKNS}, and the optimality will follow \cite{RS07}.

\section{The de Rham Complex on Approximating Manifold}
\label{sec:GeoPrelim}

\subsection{Hodge-de Rham Theory and Diffeomorphic Riemannian Manifolds}
We next introduce the Hodge-de Rham complex of differential forms on a compact oriented Riemannian manifold.  This discussion will be minimal and closely follows \cite{HoSt10a,HoSt10b}, where a more complete development can be found.

We assume $M$ is a smooth, oriented, compact $m$-dimensional manifold equipped with a Riemannian metric, $g$.   Let  $\Omega^k(M)$ be the space of smooth $k$-forms on $M$, and define the \textit{$L^2$ inner product} for any $u,v \in \Omega^k(M)$ as
\begin{equation*}
\langle u, v \rangle_{L^2\Omega(M)} = \int_M u \wedge \star_g v = \int_M \llangle u, v \rrangle_g \mu_g,
\end{equation*}
where $\star_g: \Omega^k(M) \rightarrow \Omega^{m-k}(M)$ is the Hodge star operator associated to the metric, $\llangle \cdot, \cdot \rrangle_g$ is the pointwise inner product induced by $g$, and $\mu_g$ is the Riemannian volume form.  For each $k$, define $L^2\Omega^k(M)$ as the Hilbert space formed by the completion of $\Omega^k(M)$ with respect to the $L^2$-inner product.

 Combined with the exterior derivative, $d^k:\Omega^k(M) \rightarrow \Omega^{k+1}(M)$, these spaces form a Hilbert complex, $(L^2\Omega^k(M), d)$,  with domain complex $(H\Omega^k(M), d)$.  Here $H\Omega^k(M) \subset L^2\Omega^k(M)$ are the elements in $L^2\Omega^k(M)$ with a weak exterior derivative in $L^2\Omega^{k+1}(M)$.  Each space $H\Omega^k(M)$ is endowed with a graph inner-product,

\begin{equation*}
\langle u, v \rangle_{H\Omega(M)} = \langle u, v \rangle_{L^2\Omega(M)} +   \langle du, dv \rangle_{L^2\Omega(M)},
\end{equation*}
and the complex can be described with the following diagram, 

\begin{equation}
0\rightarrow H\Omega^0(M) \xrightarrow{d}    H \Omega^{1}(M) \xrightarrow{d} \cdots  \rightarrow  H \Omega^m(M) \xrightarrow{d} 0.
\end{equation}

Next, assume $M_A$ is a polygonal, oriented, compact Riemannian manifold equipped with a metric $g_A$ and an orientation preserving differmorphism $\varphi_A: M_A \rightarrow M$.  For any point $x\in M_A$,  let $\{e_1,...,e_m\}$ and $\{f_1,...,f_m\}$ be positively-oriented orthonormal (with respect to the given metric ) bases for the tangent spaces $T_x M_A$ and  $T_{\varphi(x)}M$.  The tangent map $T_x \varphi_A: T_x M_A \rightarrow T_{\varphi(x)}M$ can be represented by an $m \times m$ matrix with $m$ strictly positive singular values independent of the choice of basis,
\begin{equation*}
\alpha_1(x) \ge \cdots \ge \alpha_m > 0.
\end{equation*}
The next theorem, from \cite{HoSt10a,Stern2013a}, describes a useful property of these singular values; see also~\cite{Ciarlet.P1978} for the classical version of the result in the case of domains in $\mathbb{R}^n$.

\begin{theorem}\label{HS41}
Let $(M_A, g_h)$ and $(M,g)$ be oriented m-dimensional Riemannian manifolds, and let $\varphi_h: M_A \rightarrow M$ be an orientation-preserving diffeomorphism with singular values $\alpha_1(x) \ge \cdots \ge \alpha_n(x) > 0$ at each $x \in M_A$.  Given $p,q \in [ 1, \infty ] $ such that $1/p + 1/q = 1$, and some $k = 0,\dots,m$, suppose that the product $( \alpha_1 \dots \alpha_k )^{1/p}$ $\cdot (\alpha_{m-k+1} \dots \alpha_m )^{-1/q}$ is bounded uniformly on $M_A$.  Then, for any $\omega \in L^p\Omega^k(M_A)$,
\begin{align*}
\| (\alpha_1 \cdots  \alpha_k)^{1/q}  (&\alpha_{k+1} \cdots  \alpha_m)^{-1/p} \|^{-1}_\infty \| \omega \|_p \\
&\le \| \varphi_{h*} \omega \|_p \le \| (\alpha_1 \cdots  \alpha_{m-k})^{1/p}  (\alpha_{m-k+1} \cdots  \alpha_m)^{-1/q} \|_\infty \| \omega \|_p.
\end{align*}
\end{theorem}

Holst and Stern then use the above theorem with $q=p=2$, noting the compactness of the manifolds yields the uniform boundedness condition, to show that $\varphi_A$ induces Hilbert complex isomorphisms $\varphi_{A*}: L^2\Omega(M_A) \rightarrow L^2\Omega(M) $ and $\varphi_A^*:L^2\Omega(M) \rightarrow L^2\Omega(M_A)$.

\subsection{Signed Distance Functions and Euclidean Hypersurfaces}

Let $M \subset \mathbb{R}^{m+1}$ be a compact, oriented, $m$ dimensional Euclidean hypersurface.  It is then possible to construct an open neighborhood, $U$, encompassing the surface with a well-defined mapping along normals to the surface, $a:U \rightarrow M$.  Furthermore, associated to any such surface is a value $\delta_0 > 0$ such that the set of points whose Riemanian distance from $M$ is less than  $\delta_0$ forms such a neighborhood.   Given an adequate $U$, let $\delta:U \rightarrow \mathbb{R}$ be the standard signed distance function.  Then for every $x\in U, \nabla \delta(x) = \nu(x)$ is the outward facing unit normal vector to the surface at $a(x)$, and 
\begin{equation*}
x = a(x) + \delta(x)\nu(x),
\end{equation*}
and the normal projection $a: U \rightarrow M$ can be expressed
\begin{equation*}
a(x)  = x-  \delta(x)\nu(x).
\end{equation*}

Thus for any approximating surface, $M_h \subset U$, the mapping $a(x)$ restricted to $M_h$ gives an orientation-preserving diffeomorphism, $\varphi_h(x)= a|_{M_h}:M_h \rightarrow M$, with well defined singular values.  Therefore Theorem \ref{HS41} can be applied.

The  approximating surface is introduced as a computational tool used to  approximate solutions to the Hodge Laplacian which are then mapped to the smooth approximated surface.  In order to do this it is necessary to develop a map of $k$-forms between the two surfaces.  In doing this we follow a subset of the ideas of \cite{HoSt10a,HoSt10b}.  Letting $P= I - v \otimes v$ and $S = - \nabla v$, we have

\begin{equation*}
\nabla a = I - \nabla \delta \otimes v - \delta \nabla v = I - v \otimes v - \delta \nabla v = P + \delta S.
\end{equation*}
This leads to the following theorem from \cite{HoSt10a} allowing for the computation of the pullback map, $a^*: \Omega^1(M) \rightarrow \Omega^1(U)$.

\begin{theorem}(Holst and Stern \cite{HoSt10a} Theorem 4.3)
Let $M$ be an oriented, compact, m-dimensional hyper surface of $\mathbb{R}^{m+1}$ with a tubular neighborhood U.  If  $Y \in T_y M$  and $x \in a^{-1}(y) \subset U$. them the lifted vector $a^*Y \in T_x U$ satisfies
\begin{equation*}
a^*Y = ( I + \delta S) Y
\end{equation*}
\end{theorem}
\begin{proof}
See \cite{HoSt10a}.
\end{proof}

Let $j:M\hookrightarrow \mathbb{R}^{m+1}$ and $j_h:M_h\hookrightarrow \mathbb{R}^{m+1}$ be inclusions of the submanifolds endowed with metrics $g= j^* \gamma$ and $g_h= j^*_h \gamma$, where $\gamma$ is the standard Euclidean metric.  For a point $x \in M_h$, the mapping can be restricted to $T_x M_h$ by composing $a^*Y$ with the adjoint of $j_h$, yielding the adjoint of the restricted tangent map $T \varphi_h = j_h^* a^*$, satisfying
\begin{equation*}
Y_h = j^*_h a^* Y = P_h ( I + \delta S ) Y.
\end{equation*}

\subsection{Discrete Problem on a Euclidean Surface}

The Hodge Laplacian defined on a Euclidean hypersurface is our main problem of interest:  Find $(\sigma, u, p)  \in H\Omega^{k-1}(M) \times H\Omega^{k}(M) \times \mathfrak{H}^k$ such that
\begin{equation}\label{surfaceHL}
\begin{array}{rll} 
\langle \sigma, \tau \rangle - \langle d \tau,  u \rangle & =  0, & \forall \tau \in H\Omega^{k-1}(M),  \\ 
\langle  d \sigma,  v \rangle + \langle d u,   d v \rangle+\langle  p,  v  \rangle & =  \langle f,  v\rangle, & \forall v \in H\Omega^{k}(M),\\ 
\langle  u,  q \rangle  &= 0, & \forall q \in \mathfrak{H}'^k .
\end{array}
\end{equation}

For the remainder of the paper, let $M_A$ be an approximating surface satisfying assumptions of the previous section.  Then $\varphi_{A*}$ and $\varphi_{A}^*$ act as the isomorphisms between  $H\Omega(M)$ and $H \Omega(M_A)$.  For ease of discussion, and similarity to the general maps in the FEEC frameworks,  we use the notation $i_A$ and $\pi_A$ respectively for $\varphi_{A*}$ and $\varphi_{A}^*$.  This notation is also consistent with the current literature in the sense that $\varphi_{A*}$ and $\varphi_{A}^*$ are an injective and projective Hilbert complex morphisms.   Since we have an isomorphism of Hilbert complexes we can define an equivalent problem on $M_A$ using the map $i_A$.  Find $(\sigma', u', p')  \in H\Omega^{k-1}(M_A) \times H\Omega^{k}(M_A) \times \mathfrak{H}'^k$ such that
\begin{equation*}
\begin{array}{rll} 
\langle i_A \sigma', i_A \tau \rangle - \langle i_A d \tau, i_A u' \rangle & =  0, & \forall \tau \in H\Omega^{k-1}(M_A),  \\ 
\langle i_A d \sigma', i_A v \rangle + \langle i_A d u',  i_A d v \rangle+\langle i_A p', i_A v  \rangle & =  \langle f, i_A v\rangle, & \forall v \in H\Omega^{k}(M_A),\\ 
\langle i_A u', i_A q \rangle  &= 0, & \forall q \in \mathfrak{H}'^k .
\end{array}
\end{equation*}

This equivalent reformulation is helpful in defining a practical discrete problem:  find $(\sigma'_{h}, u'_{h}, p'_{h})  \in V_{h}^{k-1}(M_A) \times V_{h}^k(M_A) \times \mathfrak{H'}^k_{h}$ such that
\begin{equation*}\label{rw2}
\begin{array}{rll} 
\langle i_A \sigma'_{h} , i_A \tau \rangle - \langle d ( i_A \tau), i_A u'_{h} \rangle& =  0, & \forall \tau \in V^{k-1}_{h}(M_A),  \\ 
\langle d (i_A \sigma'_{h}), i_A v \rangle + \langle  d (i_A u_{h}'),  d (i_A v) \rangle+\langle i_A p_{h}', i_A v  \rangle& =  \langle f, i_A v\rangle, & \forall v \in V^k _{h}(M_A),\\ 
\langle i_A u_{h}', i_A q \rangle  &= 0, & \forall q \in \mathfrak{H'}^k_{h} .
\end{array}
\end{equation*} 

Here $\mathfrak{H}'^k$ and $\mathfrak{H'}^k_{h}$ are the spaces  which $i_A$ maps to harmonic forms in $H\Omega^k(M)$.  The properties of these spaces will not affect our analysis, but it is worth noting that these spaces are distinct from $\mathfrak{H}^k$ and $\mathfrak{H}^k_{h}$ (see \cite{HMS} for a detailed discussion).   

Using this discrete formulation is equivalent to defining finite element spaces on the polygonal approximating surface and mapping them to $M$.  The spaces on $M_A$ can be refined with standard techniques yielding a refined mapped space.  Using this discrete formulation, we will prove a convergent and optimal algorithm for solving \eqref{surfaceHL}.  Notationally we will use $\mathcal{T}$ to represent a triangulation of the linear approximating surface, and $a(\mathcal{T})$ to represent the triangulation mapped to the approximated surface.

Unlike \cite{HMS} we are dealing with manifolds which may not have a boundary, and thus harmonics may be present in the case $k=m$.  However, in the case $k=m$, the harmonic component of $f$ on a surface without boundary is simply the constant volume form.  In this situation the harmonic component of $f$ can be calculated efficiently, essentially reducing the problem to a  $\mathfrak{B}$ problem.  Therefore we focus on $\mathfrak{B}$ problems  in the case $k=m$ for the remainder of the paper.

\section{Approximation, Orthogonality, and Stability Properties}

\subsection{Approximation Properties}
Before proceeding, we prove similar results to those of Section 2 for cases in which the finite element space is no longer constructed on triangulations of  polygonal domains in $\mathbb{R}^n$.  Here, and for the remainder of our analysis, we will use a triangulation of a polynomial approximating surface, $M_A$, and pull the spaces $\mathcal{P}_r\Omega^k(M_A)$ or $\mathcal{P}^{-}_r\Omega^k(M_A)$ to the surface $M$.

Next, we define $i_A^* : L^2 \Omega(M) \rightarrow L^2\Omega(M_A)$ as the  the adjoint of $i_A$, such that
\begin{equation*}
\langle i_A u, i_A v \rangle_M = \langle i^*_A i_A u, v \rangle_{M_A}, \hspace{.5cm}  \forall u, v \in L^2\Omega(M_A).
\end{equation*}

The $H^1$ boundedness of $i_A^*$ will be important in our convergence analysis.  We prove this boundedness in Lemma \ref{H1rLemma}, but first introduce an intermediary lemma.

\begin{lemma}\label{doubleHodgeStar}
Given $\tau \in L^2\Omega^k(M)$ we have
\begin{equation}
i_A^* \tau = (-1)^{k (m-k)} \star_{M_A} \pi_A \star_M \tau
\end{equation} 
where $\star_M$ and $\star_{M_A}$ are the Hodge star operators related to the surfaces $M$ and $M_A$.
\end{lemma}
\begin{proof}
\begin{align*}
\langle i^*_A \tau, \sigma \rangle_{M_A} &= \langle \tau , i_A \sigma \rangle_M, \\
&= \int_M \llangle \tau, i_A \sigma \rrangle \mu_M, \\
&= \int_M i_A \sigma \wedge \star_M \tau, \\
&= \int_M i_A \sigma \wedge (i_A \pi_A ) \star_M \tau
\end{align*}
Next, since the pullback commutes with the wedge product and $\star_{M_A} \star_{M_A} = (-1)^{k(m-k)}$, we have
\begin{align*}
&= \int_{M_A}  \sigma \wedge (\pi_A ) \star_M \tau, \\
&= (-1)^{k(m-k)}\int_{M_A} \sigma \wedge (\star_{M_A} \star_{M_A} ) \pi_A \star_M \tau, \\
&= (-1)^{k(m-k)}\int_{M_A} \llangle \sigma, \star_{M_A} \pi_A \star_M \tau \rrangle \mu_{M_A}, \\
&= \langle \sigma , (-1)^{k (m-k)} \star_{M_A} \pi_A \star_M \tau \rangle_{M_A}.
\end{align*}
Given the construction of the Hilbert spaces this is sufficient to complete the proof.
\end{proof}


\begin{lemma}\label{H1rLemma}
Let $\tau \in H^1\Omega^{m}(M)$, and let $i_A$ be defined as above.  Then $i^*_A \tau \in H^1 \Omega^{m} (M_A)$, and
\begin{equation}\label{H1relationship}
C_1 \| i^*_A \tau \|_{H^1 \Omega^{m}(M_A) } \le \|  \tau \|_{H^1 \Omega^{m}(M) } \le C_2 \| i^*_A \tau \|_{H^1 \Omega^{m}(M_A) }.
\end{equation}
\end{lemma}

\begin{proof}
 Lemma \ref{doubleHodgeStar} shows, in the case $k=m$, that the bounds on $i^*_A$ are the same as those used for $\pi_A$ in the case $k=0$.  In the case $k=0$, $\pi_A$ maps $H^1 \Omega(M) \rightarrow H^1 \Omega(M_A)$, with bounds introduced earlier.
\end{proof}

Next we introduce a new interpolant, $I_{M_h}$ which is related to the canonical interpolant introduced earlier.

\begin{definition}
Let $I_h$ be the canonical projection operator on $\mathcal{T}_h$, a triangulation of $M_A$.  Then for $\tau \in H\Omega^k(M)$, we define $I_{M_h}$ as
\begin{equation*}
I_{M_h} \tau = i_A I_h (\pi_A \tau )
\end{equation*}
\end{definition}

\begin{lemma}\label{canProjMapped}  Suppose $\tau \in H^1\Omega^{k}(M)$, where $k = m-1$ or $k = m $.  Let $I_{M_h}$ be the altered canonical projection operator introduced above. Then 
 \begin{align}
 I_{M_h} d &= d I_{M_h}  
 \end{align}
\end{lemma} 
\begin{proof}
Using Lemma \ref{canProj} and the fact that the pull-back and push-forward commute with $d$, we have
\begin{align*}
I_{M_h} d &= i_A I_h \pi_A d \\
&= i_A I_h d \pi_A \\
&= d (i_A I_h \pi_A).
\end{align*}
\end{proof}

\begin{lemma}
Let $f \in H \Omega^m(M)$, and let $T \in \mathcal{ T}_h$.  Then
\begin{equation}
\int_{a(T)} ( f_h - I_{M_H} f_h ) = 0
\end{equation}
\end{lemma}
\begin{proof}
From \eqref{discInt} we know
\begin{equation*}
\int_T ( I_h( \pi_A f ) - I_H (I_h \pi_A f )) = 0
\end{equation*}
and since $k=m$, we have 
\begin{equation*}
 \int_{a(T)} ( i_A I_h( \pi_A f ) - i_A I_H (I_h \pi_A f )) = 0
\end{equation*}
yielding
\begin{equation*}
 \int_{a(T)}  I_{M_h} f  - (i_A I_H \pi_A )I_{M_h} f   = 0
\end{equation*}

\end{proof}

\subsection{Quasi-Orthogonality}
\label{sec:quasi}

The main difficulty for mixed finite element methods is the lack of minimization principle, and thus the failure of orthogonality.  In \cite{HMS} results from \cite{CHX} are generalized, and a quasi-orthogonality property is proven using the fact that $\sigma - \sigma_h$ is orthogonal to the subspace $\mathfrak{Z}^{n-1}_h \subset H\Lambda^{n-1}_h(\Omega)$.  In this section we show that this same orthogonality result holds for finite elements spaces mapped to the smooth surface from the approximating surface. 

Solutions of Hodge Laplace problems on nested triangulations $\mathcal{T}_h$ and $\mathcal{T}_H$ will frequently be compared.  Nested in the sense that $\mathcal{T}_h$ is a refinement of $\mathcal{T}_H$.  For a given $f \in L^2 \Omega^m(M)$, let $\mathcal{L}^{-1}f$ denote the  solutions of \eqref{surfaceHL}.  Let $\mathcal{L}^{-1}_h f_h$ and $\mathcal{L}^{-1}_H f_H$  denote the solutions to the discrete problems on $a(\mathcal{T}_h)$ and $a(\mathcal{T}_H)$ respectively.  Set the following triples, $( u, \sigma, p )= \mathcal{L}^{-1}f$, $( u_h, \sigma_h, p_h )= \mathcal{L}^{-1}_h f_h$, $( \tilde{u}_h, \tilde{\sigma}_h, \tilde{p}_h )= \mathcal{L}^{-1}_h f_H$ and $( u_H, \sigma_H, p_H )= \mathcal{L}^{-1}_H f_H$.  The following analysis deals with the $\mathfrak{B}$ problem and thus the harmonic component will be zero in each of these solutions.  When we are only interested in $\sigma$ we will abuse this notation by writing $\sigma = \mathcal{L}^{-1} f $.

\begin{lemma}   
 Given $f \in L^2 \Omega^m(M)$ in $\mathfrak{B}$, and two nested triangulations $\mathcal{T}_h$ and $\mathcal{T}_H$, then
\begin{equation}\label{nestTri}
\langle \sigma - \sigma_h, \tilde{\sigma}_h - \sigma_H \rangle_M = 0.
\end{equation}

\end{lemma}

\begin{proof}
See \cite{HMS}.
\end{proof}

The next result is similar to Theorem 4.2 in \cite{HMS}.  We present the proof in order to clarify the impact of the surface mapping.


\begin{theorem}\label{quasi0}  Given $f \in L^2 \Omega^m(M)$ in  $\mathfrak{B}$, and two nested triangulations $\mathcal{T}_h$ and $\mathcal{T}_H$, then

\begin{equation}\langle \sigma - \sigma_h, \sigma_h - \sigma_H \rangle \le \sqrt{C_0}\|\sigma - \sigma_h \| \textit{osc}( \pi_A f_h,  \mathcal{T}_H), \end{equation}

and for any $\delta > 0$,
\begin{equation}
(1-\delta)\|\sigma - \sigma_h\|^2 \le \| \sigma - \sigma_H \|^2 - \| \sigma_h - \sigma_H \|^2 + \frac{C_0}{\delta}\textit{osc}^2( \pi_A f_h, \mathcal{T}_H ). \label{chx32}
\end{equation}

\end{theorem}

\begin{proof}
By \eqref{nestTri} we have
\begin{align*}
\langle \sigma - \sigma_h, \sigma_h - \sigma_H \rangle &= \langle \sigma - \sigma_h, \sigma_h - \tilde{\sigma}_h \rangle + \langle \sigma - \sigma_h, \tilde{\sigma_h} - \sigma_H \rangle \\
&= \langle \sigma - \sigma_h, \sigma_h - \tilde{\sigma}_h \rangle \\
&\le \| \sigma - \sigma_h \| \| \sigma_h - \tilde{\sigma}_h \| .
\end{align*}

And then by the discrete stability result, Theorem \ref{dStab}, we have
\begin{equation*}
\le \sqrt{C_0}\|\sigma - \sigma_h \| \textrm{osc}( \pi_A f_h, \mathcal{T}_H ).
\end{equation*}
\eqref{chx32} follows standard arguments and is identical to \cite{CHX} (3.4)
\end{proof}

\subsection{Continuous and Discrete Stability}
\label{sec:stab}

In this section we will prove stability results for approximate solutions to the $\sigma$ portion of the Hodge Laplace problem.  Theorem \ref{contS} gives a stability result for particular solutions of the Hodge de Rham problem that will be useful in bounding the approximation error in Section \ref{sec:bounds}.  Theorem \ref{dStab} will prove the discrete stability result used in Theorem \ref{quasi0}.  These proofs follow the same structure as \cite{HMS}, with additional steps that take care of the mapping between the surfaces.

\begin{theorem}\label{contS}\textit{ (Continuous Stability Result)}
Given $f \in L^2 \Omega^m(M)$ in $\mathfrak{B}$, let $\mathcal{T}_h$ be a triangulation of $M_A$.  Set $(\sigma, u, p ) = \mathcal{L}^{-1}f$ and $(\tilde{\sigma}, \tilde{u}, \tilde{p} ) = \mathcal{L}^{-1}f_h$,  then

\begin{equation} \| \sigma - \tilde{\sigma} \| \le C \textit{osc}( \pi_A f, \mathcal{T}_h). \end{equation}
 
\end{theorem}

\begin{proof}
The harmonic terms are vacuous,  thus
\begin{equation*}
\|\sigma - \tilde{\sigma} \|_M^2 = \langle d(\sigma - \tilde{\sigma}) , u - \tilde{u} \rangle_M = \langle f - f_h, u - \tilde{u} \rangle_M = \langle \pi_A(f - f_h), i^*_A ( u - \tilde{u} ) \rangle_{M_A}.
\end{equation*}
Let $v = u-\tilde{u}$.  Since $v \in \mathfrak{B}^k $ and $\| \delta v \| = \|$grad $v \| = \| \sigma - \tilde{\sigma} \| $, we have $v \in H^1 \Omega^m(M)$.  Restricting $v$ to an element $a(T) \in a(\mathcal{T}_h)$, we have $v \in H^1 \Omega^m(a(T))$, thus
\begin{align*}
\| \sigma -\tilde{\sigma} \|^2  = \langle f - f_h, v \rangle &=\displaystyle\sum\limits_{T \in \mathcal{T}_h}  \langle \pi_A (f - f_h) , i_A^* v  \rangle_{T}. \\
&= \displaystyle\sum\limits_{T \in \mathcal{T}_h}  \langle \pi_A  f -  \pi_A f_h , i_A^* v - I_h (i_A^* v )   \rangle_{T}. 
\end{align*}
Applying \eqref{h1interp} and then \eqref{H1rLemma},
\begin{align*}
&\le C \displaystyle\sum\limits_{T \in \mathcal{T}_h } h_T \|\pi_A f- \pi_A f_h \|_{T} \| i_A^* v\|_{H^1\Lambda^n(T)} \\
&\le C \displaystyle\sum\limits_{T \in \mathcal{T}_h } h_T \|\pi_A f- \pi_A f_h \|_{T} \|  v\|_{H^1\Lambda^n(a(T))} \\
&= C \displaystyle\sum\limits_{T \in \mathcal{T}_h}  h_T \|\pi_A f- \pi_A f_h \|_{T}( \| u - \tilde{u} \|_{a(T)} +  \| \delta( u - \tilde{u} )  \|_{a(T)} ) \\
&\le C (\displaystyle\sum\limits_{T \in \mathcal{T}_h}   \|h_T (\pi_A f- \pi_A f_h) \|^2_{T} )^{1/2} (\displaystyle\sum\limits_{T \in T_h} ( \| u - \tilde{u} \|_{a(T)} +  \| \delta( u - \tilde{u} )  \|_{a(T)} )^2)^{1/2},
\end{align*}
and $v \in H^1\Omega^m(M)$ allows us to to combine terms of the summation,
\begin{equation*}
\le C (\displaystyle\sum\limits_{T \in \mathcal{T}_h}   \|h_T (\pi_A f- \pi_A f_h) \|^2_{T} )^{1/2} ( \| u - \tilde{u} \|_M +  \| \delta( u - \tilde{u} )  \|_M ).
\end{equation*}
Since $u - \tilde{u} \in \mathfrak{B}^k$, $\| u - \tilde{u} \| = \langle u- \tilde{u}, d\tau \rangle$ for some $\tau \in \mathfrak{Z}^{\perp}$ with $\|d\tau \| = 1$, thus
\begin{equation*}
= C (\displaystyle\sum\limits_{T \in \mathcal{T}_h}  \|h_T (\pi_A f- \pi_A f_h) \|^2_{T} )^{1/2} \langle(  \sigma - \tilde{\sigma} )  , \tau \rangle_{M} +  \|  \sigma - \tilde{\sigma}  \|_M).
\end{equation*}
Then applying Poincar\'e on $\tau$:
\begin{equation*}
= C\|  \sigma - \tilde{\sigma}  \| (\displaystyle\sum\limits_{T \in \mathcal{T}_h}  \|h_T (\pi_A f- \pi_A f_h) \|^2_{T} )^{1/2}  .
\end{equation*}
Divide through by $\|  \sigma - \tilde{\sigma}  \|$ to complete proof.
\end{proof}

$\linebreak$
The following is Lemma 4 in \cite{DH}, and is a special case of Theorem 1.5 of \cite{MMM}.  It is related to the bounded invertibility of $d$, and will be an important tool in proving discrete stability.

\begin{lemma}\label{DHL4}  Assume that B is a bounded Lipschitz domain in $\mathbb{R}^n$ that is homeomorphic to a ball.  Then the boundary value problem d$\varphi = g \in  L_2\Lambda^k(B)$ in $B, \textit{tr } \varphi = 0$ on $\partial B $ has a solution $\varphi \in H^1_0\Lambda^{k-1}(B) $ with $\| \varphi\|_{H^1\Lambda^{k-1}(B) } \le C\|g\|_B$ if and only if dg = 0 in B, and in addition, \textit{tr} g = 0 on $\partial$B if 0 $\le k \le n-1$ and $\int_Bg =0$ if $k=n$ .
\end{lemma}


The next lemma is an intermediate step in proving the discrete stability result.  The general structure follows \cite{CHX} and applies Lemma \ref{DHL4} in order to find a sufficiently smooth function that is essentially a bounded inverse of $d$ for the approximation error of $u_h$ on $\mathcal{T}_H$.

In \cite{HMS}, Lemma \ref{PL1} is applied to shape regular polygonal elements, and thus the multiplicative constant can be bounded.  In this case, however, the elements are not necessarily polygonal, and thus we are forced to map the proof to $M_A$, where the regularity is clear, and then map back to $M$. 

\begin{lemma}\label{umpu} Let $\mathcal{T}_h$, $\mathcal{T}_H$ be nested conforming triangulations and let $\sigma_h, \sigma_H$ be the respective solutions to \eqref{DHL} with data $f \in L^2\Omega^m(M)$ in $\mathfrak{B}$. Then for any $T \in \mathcal{T}_H$    

\begin{equation} \| I_h i^*_A u_h -I_{H} i^*_A u_h \|_{T} \le \sqrt{C_0}  h_T \|  \sigma_h \|_{a(T)}. \end{equation}
\end{lemma}

\begin{proof}
Let  $g_\Omega =  I_h i^*_A u_h - I_{H} i^*_A u_h = ( I_{h} - I_{H} ) i^*_A  u_h  \in L^2 \Omega^m(M_A)$.
Then, for any $T \in  \mathcal{T}_H$  let $g$ =  tr$_{T} g_\Omega \in L^2 \Omega^m(T)$, and by Lemma \ref{PL1}, $ \int_{T} g = 0$.  Thus Lemma \ref{DHL4} can be applied to
 find $\tau \in H^1_0\Lambda^{n-1}(T)$, such that:
\begin{align*}
d   \tau  &=  ( I_{h} - I_{H}) i_A^* u_h, \textrm{ on } T \\
\| \tau\|_{H^1 \Lambda^{n-1}(T)} &\le C\|(I_{h} - I_{H}) i^*_A u_h \|_{T}.
\end{align*}
Extend $\tau$ to $H^1 \Lambda^{n-1}(\Omega)$ by zero and then, by Lemma \ref{PL2},
\begin{equation*}
\| (I_{h} - I_{H}) i^*_A u_h \|^2_{T} = \langle ( I_{h} - I_{H} ) i^*_A  u_h , d  \tau \rangle_T =\langle   i^*_A u_h , d ( I_{h} - I_{H} )  \tau \rangle_T  
 \end{equation*}
Then by Lemma \ref{canProj}, and locality of $\tau$,
\begin{equation*}
= \langle i^*_A u_h ,  d ( I_{h} - I_{H} )  \tau \rangle_{M_A} = \langle  u_h ,  d ( i_A( I_{h} - I_{H} )  \tau ) \rangle_M =\langle \sigma_h ,  i_A ( I_{h} - I_{H} )  \tau \rangle_M.
\end{equation*}
Then again by locality of $\tau$ and Theorem \ref{HS41}, 
\begin{align*}
= \langle  \sigma_h , i_A ( I_{h} - I_{H} ) \tau \rangle_{a(T)} &\le \| \sigma_h \|_{a(T)} ( \|i_A( \tau - I_h  \tau ) \|_{a(T)} + \| i_A( \tau - I_H  \tau )\|_{a(T)} ),\\
 &\le \|  \sigma_h \|_{a(T)} ( \| \tau - I_h  \tau  \|_{T} + \|  \tau - I_H  \tau \|_{T} ).
\end{align*}
And by \eqref{h1interp},
\begin{equation*}
\le C h_T \|  \sigma_h\|_{a(T)} \|  \tau \|_{H^1(T)} \le C h_T \|  \sigma_h \|_{a(T)} \| (I_{h} - I_{H} ) i^*_A u_h \|_{T}.
\end{equation*}
Cancel one power of $\| (I_{M_h} - I_{M_H}) u_h \|_{T} $  to complete the proof.

\end{proof}

\begin{theorem}\label{dStab}\textit{(Discrete Stability Result)} Let $\mathcal{T}_h$ and $\mathcal{T}_H$ be nested conforming triangulations.  Let $( \tilde{u}_h, \tilde{\sigma}_h, \tilde{p}_h ) =  \mathcal{L}_h^{-1}f_H$ and $(u_h, \sigma_h, p_h ) = \mathcal{L}_h^{-1} f_h$, with $f \in L^2\Omega^m(M)$ in $\mathfrak{B}$.
Then there exists a constant such that

\begin{equation} \| \sigma_h - \tilde{\sigma}_h \| \le C \textit{osc}( \pi_A f_h, \mathcal{T}_H ) \end{equation}

\end{theorem}

\begin{proof}

From \ref{DHL}, and since $p_h, \tilde{p}_h = 0$, we have
\begin{align}
\label{steq1} \langle \sigma_h - \tilde{\sigma_h }, \tau_h \rangle &= \langle u_h - \tilde{ u}_h, d \tau_h \rangle, \quad \forall \tau_h \in \Lambda_h^{k-1}, \\ 
\label{steq2} \langle d(\sigma_h - \tilde{\sigma_h }), v_h \rangle &= \langle f_h - f_H,  v_h \rangle, \quad \forall v_h \in \Lambda_h^{k}.
\end{align}
Next set $\tau_h = \sigma_h - \tilde{\sigma}_h$ in \eqref{steq1}, and $v_h = u_h - \tilde{u}_h $ in \eqref{steq2} to obtain:
\begin{equation*}
\| \sigma_h - \tilde{\sigma}_h \|^2 = \langle u_h - \tilde{u}_h, d( \sigma_h - \tilde{\sigma}_h ) \rangle = \langle f_h - f_H,  v_h \rangle,
\end{equation*}

Then by Lemma \ref{umpu}, we have:
\begin{align*}
\| \sigma_h - \tilde{\sigma}_h \|^2 &= \displaystyle\sum\limits_{T \in \mathcal{T}_H}   \langle  v_h,  f_h -  f_H \rangle_{a(T)} \\
&= \displaystyle\sum\limits_{T \in \mathcal{T}_H}   \langle  i^*_A v_h, I_h \pi_A f -  I_H \pi_A f \rangle_{T} \\
&= \displaystyle\sum\limits_{T \in \mathcal{T}_H}   \langle  I_h (i^*_A v_h ), I_h \pi_A f -  I_H \pi_A f \rangle_{T} \\
&= \displaystyle\sum\limits_{T \in \mathcal{T}_H}   \langle I_h( i^*_A v_h ) - I_H (i^*_A v_h ), I_h \pi_A f - I_H \pi_A f \rangle_{T} \\
 & \le C \displaystyle\sum\limits_{T \in \mathcal{T}_H}  \|  I_h \pi_A f -  I_H  \pi_A f \|_{T} \|  I_h  i^*_A v_h - I_H i^*_A v_h \|_{T}  \\
 &\le C \displaystyle\sum\limits_{T \in \mathcal{T}_H}  h_T \| I_h \pi_A f -  I_H \pi_A f \|_{T} \|  ( \sigma_h - \tilde{\sigma}_h ) \|_{a(T)} \\
&\le C( \displaystyle\sum\limits_{T \in \mathcal{T}_H}  h_T^2  \| I_h \pi_A f -  I_H \pi_A f \|_{T}^2  )^{1/2} \| \sigma_h - \tilde{\sigma}_h \|_M
\end{align*}

Then cancel one $  \| \sigma_h - \tilde{\sigma}_h \| $ to complete the proof.

\end{proof}

\section{{\em A Posteriori} Error Indicator and Bounds}
\label{sec:bounds}

In this section we introduce the \textit{a posteriori} error estimators used in our adaptive algorithm.  The estimator follows from \cite{HMS} which follows \cite{ AA, CHX}.  The difference here is that we estimate the error on the fixed approximating surface.  Next, applying ideas \cite{HMS} to the surface estimator, we prove bounds on these estimators and a continuity result, both of which are key ingredients in showing the convergence and optimality of our adaptive method.

\subsection{Error Indicator: Definition, Lower bound and Continuity}
\begin{definition}\textit{(Element Error Indicator)}  Let $T \in \mathcal{T}_H, f \in L^2\Omega^m(M)$ in $\mathfrak{B}$, and $\sigma_H = \mathcal{L}^{-1} f_{H}$.  Let the jump in $\tau$ over an element face be denoted by $[[\tau]]$.  For element faces on $\partial \Omega$ we set $[[\tau]] = \tau$. The element error indicator is defined as
\begin{equation*}
\eta_T^2(\sigma_H) =  h_T \| [[ \textit{tr }  \star (\pi_A \sigma_H) ]] \|_{\partial T }^2 +  h_T^2\|  \delta  (\pi_A \sigma_H) \|^2_{T } +  h_T^2\|  \pi_A f - d (\pi_A \sigma_H) \|^2_{T }
\end{equation*}
\vspace{.2 cm }
For a subset $\mathcal{\tilde{T}}_H \subset \mathcal{T}_H$, define
\begin{equation*}
\eta^2( \sigma_H,\mathcal{\tilde{T}}_H ) := \displaystyle\sum\limits_{T \in \mathcal{\tilde{T}}_H}  \eta_T^2(\sigma_H)  
\end{equation*}
\end{definition}

\begin{theorem}\textit{(Lower Bound)}  Given $f \in L^2\Omega^m(M)$ in $\mathfrak{B}$  and a shape regular triangulation $\mathcal{T}_H$, let $\sigma = \mathcal{L}^{-1}f$ and $\sigma_H = \mathcal{L}_H^{-1}f_{H}$.  Then there exists a constant dependent only on the shape regularity of $\mathcal{T}_H$ and the surface mapping, such that
\begin{equation}
\label{lowerBS} C_2 \eta^2( \sigma_H, \mathcal{T}_H ) \le \| \sigma - \sigma_H \|^2+ C_2 \textit{osc}^2(\pi_A f,\mathcal{T}_H ) . \end{equation}
\end{theorem}
\begin{proof}
In proving a lower bound, in \cite{DH} it is shown that 
\begin{align*}
h_T \| \delta \sigma_H \|_{T} &\le C \| \sigma - \sigma_H \|_{T }, \\
h_T^{1/2} \| [[ \text{tr} \star \sigma_H ]] \|_{\partial T} &\le C \| \sigma - \sigma_H \|_{\mathcal{T}_t}, 
\end{align*}
 where $\mathcal{T}_t$ is the set of all triangles sharing a boundary with $T$.  The first is equation (5.7) and the second is a result of equation (5.12) in \cite{DH}.   Substituting $\pi_A \sigma_H$ for $\sigma$, noting that $\sigma \in \mathfrak{Z}^\perp (M)$ implies $\pi_A \sigma \in \mathfrak{Z}^\perp (M_A)$, similar results for $\pi_A \sigma_H$ follow \cite{DH}.  Then, using the boundedness of $i_A$, the remainder of the proof is identical to \cite{HMS}.
 \end{proof}

The following lemma will be important in proving a continuity result used in showing convergence of our adaptive algorithm.  It is nearly identical to an estimator efficiency proof in ~\cite{DH}, but the subtle difference is that we make use of $\sigma_H$, the solution on the less refined mesh, and $\sigma$ is not used in our arguments.

\begin{lemma} Given $f \in L^2\Omega^m(M)$ in $\mathfrak{B}$ and nested triangulations $\mathcal{T}_h$ and $\mathcal{T}_H$, let $\sigma_h = \mathcal{L}_h^{-1} f_{h}$ and $\sigma_H = \mathcal{L}_H^{-1} f_{H}$.  Then for $T \in \mathcal{T}_h$
 \begin{equation}\label{disEff}
 C_2 \displaystyle\sum\limits_{T \in \mathcal{T}_h}  (h_T \| [[ \textit{tr }  \star (\pi_A \sigma_h- \pi_A \sigma_H) ]] \|_{\partial T }^2 +  h_T^2\|  \delta (\pi_A \sigma_h - \pi_A \sigma_H) \|^2_{T} )\le \| \pi_A \sigma_h - \pi_A \sigma_H \|^2 .
 \end{equation} 
\end{lemma}

\begin{proof}
Follows \cite{HMS}.
\end{proof}

\begin{theorem}\textit{(Continuity of the Error Estimator)}
Given $f \in L^2\Omega^m(M)$ in $\mathfrak{B}$ and nested triangulations $\mathcal{T}_h$ and $\mathcal{T}_H$, let $\sigma_h = \mathcal{L}_h^{-1}f_{h}$ and $\sigma_H = \mathcal{L}_H^{-1}f_{H}$. Then we have:

\begin{equation} \beta ( \eta^2( \sigma_h, \mathcal{T}_h ) - \eta^2( \sigma_H, \mathcal{T}_h ) ) \le \| \pi_A  \sigma_h - \pi_A \sigma_H \|^2 + \textit{osc}^2( \pi_A f_{h}, \mathcal{T}_H)\label{conee} \end{equation}
\end{theorem}
\begin{proof}
Follows \cite{HMS}.
\end{proof}
\subsection{Continuous and Discrete Upper Bounds}

The following proofs have a similar structure to the continuous and discrete upper bounds proved in \cite{AA,CHX}.  A key element of the proof will be comparisons between the discrete solution $\sigma_H = \mathcal{L}_H^{-1} f_{H} $ and the solution to the intermediate problem, $\tilde{\sigma} = \mathcal{L}^{-1} f_{H}$.
We begin by looking the orthogonal decomposition of $ \tilde{\sigma} - \sigma_H $,
\begin{equation*}
 \tilde{\sigma} - \sigma_H = (\tilde{\sigma} - P_{\mathfrak{Z}^\perp }\sigma_H ) - P_{\mathfrak{B}^{k-1}} \sigma_H - P_{\mathfrak{H}^{k-1}} \sigma_H
\end{equation*}
which allows the norm to be rewritten
\begin{equation*}
\| \tilde{\sigma} - \sigma_H \|^2 = \| (\tilde{\sigma} - P_{\mathfrak{Z}^\perp }\sigma_H )\|^2 + \| P_{\mathfrak{B}^{k-1}} \sigma_H \|^2 + \| P_{\mathfrak{H}^{k-1}} \sigma_H \|^2.
\end{equation*}
Lemmas \ref{nullUBl}, \ref{bUBl} and \ref{harmonicUBl} will each bound a portion of this orthogonal decomposition.  Then Theorem \ref{CUB} will combine these results in proving the desired error bound.
\begin{lemma}\label{nullUBl}  Given an $f \in L^2\Omega^m(M)$ in $\mathfrak{B}$. Let $\tilde{\sigma} = \mathcal{L}^{-1} f_{H}$ and $\sigma_H = \mathcal{L}_H^{-1} f_{H}$.  Then
\begin{equation}\label{nullUB}
\| (\tilde{\sigma} - P_{\mathfrak{Z}^\perp }\sigma_H )\|^2 = 0.
\end{equation}
\end{lemma}
\begin{proof}
See \cite{HMS}.
\end{proof}

The next lemma uses the quasi-interpolant $\Pi_H$ described in \cite{DH}, and also applies integration by parts in the same standard fashion that \cite{DH} use when bounding error measured in the natural norm, $\|u - u_h\|_{H\Lambda^k(\Omega)} + \|\sigma - \sigma_h\|_{H\Lambda^{k-1}(\Omega)} + \|p - p_h\|$.  In~\cite{DH}, coercivity of the bilinear-form is used to separate components of the error, whereas here we simply analyze the orthogonal decomposition of $\sigma - \sigma_H$.  In~\cite{DH}, the Galerkin orthogonality implied by taking the difference between the continuous and discrete problems is employed in order to make use of $\Pi_h$.  Here we are able to introduce the quasi-interpolant by simply using the fact that $\sigma_H \perp \mathfrak{B}_H^{k-1}$.

\begin{lemma}\label{bUBl} Given an $f \in L^2\Omega^m(M)$ in $\mathfrak{B}$. Let $\sigma_H = \mathcal{L}_H^{-1} f_{H}$.  Then
\begin{equation}\label{bUB}
\| P_{\mathfrak{B}^{k-1}} \sigma_H \|^2 \le  C \eta^2(\sigma_H, \mathcal{T}_H ).
\end{equation}
\end{lemma}
\begin{proof} 
Follow the same steps as \cite{HMS} using $\| \pi_A (P_{\mathfrak{B}^{k-1}} \sigma_H) \|^2$.  Next use the boundedness of $\pi_A$ to relate to  $\| P_{\mathfrak{B}^{k-1}} \sigma_H \|^2$.
\end{proof}

\begin{lemma}\label{harmonicUBl} Given an $f \in L^2\Lambda^k(\Omega)$ in $\mathfrak{B}^k$. Let $\tilde{\sigma} = \mathcal{L}^{-1} f_{H}$ and $\sigma_H = \mathcal{L}_H^{-1} f_{H}$.  Then
\begin{equation}\label{harmonicUB}
\|P_{\mathfrak{H}^{k-1}} \sigma_H \|^2 \le C \| \tilde{\sigma} - \sigma_H \|^2, \qquad C < 1.
\end{equation}
\end{lemma}
\begin{proof} 
See \cite{HMS}.
\end{proof}

$\linebreak$
Now we have the tools to prove the continuous upper bound for the $\mathfrak{B}$ problems.
\begin{theorem}\label{CUB}$\textit{(Continuous Upper-Bound)}$  Given an $f \in L^2\Omega^m(M)$ in $\mathfrak{B}^m$. Let $\tilde{\sigma} = \mathcal{L}^{-1} f_{H}$ and $\sigma_H = \mathcal{L}_H^{-1} f_{H}$.  Then
\begin{equation} \| \sigma - \sigma_H \|^2 \le C_1 \eta^2(\sigma_H, \mathcal{T}_H ) . \label{CUBeq}\end{equation}
\end{theorem}
\begin{proof} 
See \cite{HMS}.
\end{proof}

\begin{theorem}\label{DUB}$\textit{(Discrete Upper-Bound)}$  Given $f \in L^2\Omega^m(M)$ in $\mathfrak{B}$ and nested triangulations $\mathcal{T}_h$ and $\mathcal{T}_H$, let $\sigma_h = \mathcal{L}_h^{-1} f_{h}$ and $\sigma_H = \mathcal{L}_H^{-1} f_{H}$.  Then
\begin{equation}\label{DUBeq} \| \sigma_h - \sigma_H \|^2 \le C_1 \eta^2(\sigma_H, T_H ). \end{equation}
\end{theorem}

\begin{proof}
The proof requires the same ingredients needed to prove the continuous upper bound.  
The same intermediate steps are taken by performing analysis on the $W^{k-1}_h$ orthogonal decomposition of $  \tilde{\sigma}_h - \sigma_H  $.
\begin{equation*}
\tilde{\sigma}_h - \sigma_H = (\tilde{\sigma}_h - P_{\mathfrak{Z}^\perp_h }\sigma_H ) - P_{\mathfrak{B}_h^{k-1}} \sigma_H - P_{\mathfrak{H}_h^{k-1}} \sigma_H .
\end{equation*}
The discrete version of Lemma \ref{nullUBl} uses $\delta_h$ rather than $\delta$, but is otherwise identical.  The discrete version of Lemma \ref{bUBl} is identical.  The discrete version of Lemma \ref{harmonicUBl} follows the same structure but makes use of Corollary \ref{dhdc}.  The final step in the proof uses the discrete stability result, Theorem \ref{dStab}.

\end{proof}

\section{Convergence of AMFEM}
\label{sec:conv}

After presenting the adaptive algorithm, the remainder of this section proves convergence and then optimality.  The results in this section follow ideas already in the literature \cite{ RS07, MS1, stars, DORFLER96, CHX}, with Theorem \ref{termination} following \cite{HMS} in proving reduction in a quasi-error using relationships between data oscillation and the decay of a second type of quasi-error.  The following algorithm and analysis of convergence deal specifically with the case $k=m$.  In presenting our algorithm we replace $h$ with an iteration counter $k$.

$\linebreak$
\textsf{Algorithm}:[$\mathcal{T}_N, \sigma_N$ ] =  \textsf{AMFEM}$( \mathcal{T}_0, f, \epsilon, \theta )$:  Given a fixed approximating surface of $M$, an initial shape-regular triangulation $\mathcal{T}_0$, and a marking parameter $\theta$, set $k= 0$ and iterate the following steps until a desired decrease in the error-estimator is achieved:

\begin{align*}
&(1) (u_k, \sigma_k, p_k) = SOLVE(f, \mathcal{T}_k )\\
&(2)  \{\eta_{T}\} = ESTIMATE( f, \sigma_k, \mathcal{T}_k)\\
&(3)  \mathcal{M}_k = MARK( \{ \eta_T \}  , \mathcal{T}_k, \theta) \\
&(4) \mathcal{T}_{k+1} = REFINE( \mathcal{T}_{k} , \mathcal{M}_k ) 
\end{align*}

\subsection{Convergence of AMFEM}
  The following notation will be used in the proofs and discussion of this section:
\begin{equation*}
e_k=\|\sigma- \sigma_k\|^2,
\qquad
E_k=\|\sigma_{k+1}- \sigma_k\|^2, 
\qquad 
\eta_k =\eta^2(\sigma_k, \mathcal{T}_k ),
\end{equation*}
\begin{equation*}
o_k = \text{osc}^2(f,\mathcal{T}_k), 
\qquad
\hat{o}_k = \text{osc}^2(f_{k+1},\mathcal{T}_k),
\end{equation*}
where $f_k = P_k f = P_{\mathfrak{B}_k} f$ since $k=m$.

\begin{lemma}\label{CHX61}
\begin{equation}
\beta\eta_{k+1} \le \beta( 1 -\lambda \theta) \eta_k + E_k + \hat{o}_k.
\end{equation}
\end{lemma}
\begin{proof}
This follows from continuity of the error estimator \eqref{conee}, and properties of the marking strategy, i.e. reduction of the summation on a finer mesh due to smaller element sizes on refined elements.  The proof can be found in \cite{CHX}.  $\lambda < 1$ is a constant dependent on the dimensionality of the problem. 
\end{proof}

For convenience, we recall the quasi-orthogonality \eqref{chx32} the continuous upper-bound \eqref{CUBeq} equations, 
\begin{align*}
(1-\delta)e_{k+1} &\le e_k - E_k + C_0\hat{o}_k, \text{ for any } \delta > 0, \\
e_k &\le C_1\eta_k.
\end{align*}

With these three ingredients, basic algebra leads to the following result,
\begin{theorem}\label{CHX62}
When
\begin{equation}
0 < \delta < \textit{min}\{ \frac{ \beta}{2C_1}\theta, 1\}, 
\end{equation}
there exists $\alpha \in$ \textit{(0,1)} and $C_\delta$ such that
\begin{equation}\label{alpha}
(1- \delta)e_{k+1} + \beta\eta_{k+1} \le \alpha [(1-\delta)e_k + \beta\eta_k ] + C_{\delta}\hat{o}_k.
\end{equation}
\end{theorem}
\begin{proof}
Follows the same steps as \cite{CHX}.
\end{proof}
With the above result we next prove convergence.
\begin{theorem}\textit{(Termination in Finite Steps)}\label{termination}
Let $\sigma_k$ be the solution obtained in the kth loop in the algorithm AMFEM, then for any $0<\delta < \textit{min}\{ \frac{\beta}{2C_1}\theta, 1 \}$, there exists positive constants $C_\delta$ and 0$< \gamma_\delta < 1$ depending only on given data and the initial grid such that,
\begin{equation*}
(1-\delta)\| \sigma - \sigma_k \|^2 + \beta \eta^2(\sigma_k, \mathcal{T}_k ) + \zeta \textit{osc}^2( f, \mathcal{T}_k ) \le C_q\gamma^k_\delta,
\end{equation*}
and the algorithm will terminate in finite steps.
\end{theorem}

\begin{proof}
See \cite{HMS}.
\end{proof}

\subsection{Optimality of AMFEM}
Once Theorem \ref{contS}, Theorem \ref{CHX22}, \eqref{DUBeq}, \eqref{chx32} and \eqref{lowerBS} are established, optimality can be proved independent of dimension following the proof of Theorem 5.3 in \cite{RS07}.

\begin{theorem}\textit{(Optimality)}
For any  $f \in L^2\Omega^m(M)$ in $\mathfrak{B}$, shape regular $\mathcal{T}_0$ and $\epsilon > 0$, let $\sigma = \mathcal{L}^{-1} f$ and [ $\sigma_N, \mathcal{T}_N$ ] = \textit{\textsf{AMFEM}}($\mathcal{T}_H, f_H, \epsilon /2, \theta$).  Where [$\mathcal{T}_H, f_H$] = \textit{\textsf{APPROX}}(f, $\mathcal{T}_0,\epsilon/2$).  If $\sigma \in \mathcal{A}^s$ and $f \in \mathcal{A}^s_o$, then
\begin{equation}\label{optimalityHMS}
\| \sigma - \sigma_N \| \le C( \|\sigma \|_{\mathcal{A}^s} + \| f \|_{\mathcal{A}^s_o})( \# \mathcal{T}_N - \# \mathcal{T}_0 )^{-s}.
\end{equation}
\end{theorem}

\begin{proof}
Follows directly from \cite{CHX}.
\end{proof}

The key components in the optimality of a method are the rate of the convergence and the decaying constant, and \eqref{optimalityHMS} is a good model equation for analysis.  Placing no restrictions on node placement substantially increases degrees of freedom and computational cost to the subspace approximation.  We thus restrict our discussion of optimality to the two basic cases; evolving surfaces with element nodes lying on the approximated surface, and the scheme used above. 

The rate rate of decay, $s$, is an intrinsic property related to a functions approximation class for a given refinement method.  The map between surfaces is a Hilbert complex isomorphism, and thus $H\Omega^k(M)$ will be mapped to $H\Omega^k(M_A)$, analogous to mapping between similar Sobelev spaces.  For example, when $k = m-1$ elements in  $H(div)$ on $M$ will be mapped to $H(div)$ on $M_A$.  Also, since the mapping is smooth between the two surfaces, preserving the differentiability properties of the forms.  The relationship between the smoothness of the solution and data $f$ to their approximation class is discussed in \cite{BBD04,BDDS02}.

Next we look at the multiplicative constant.  One advantage of the evolving surface approximation is that the multiplicative bounds in Theorem \ref{HS41} improve with better surface approximations.  If the initial surface approximation is good, however, this constant is negligible in terms of computation cost.  Other inefficiencies may arise by building the initial surface approximation without much analysis of the PDE.  Interpolation of the surface can be done in a standard efficient manner, and as long as the initially surface isn't excessively precise, then  the impact on $C$ in \eqref{optimalityHMS} will not be significant.  The other portion of the multiplicative constant is related to the norm of $f$ and $\sigma$ mapped to the approximating surface, and this value should be reasonable by the same arguments used for the rate of decay.

\section{Conclusion and Future Work}
\label{sec:sConclusion}
Surface finite element methods, in their nature, have additional complexities which introduce difficulties developing a generic adaptive algorithm.  Surfaces, for instance, can be described in different manners and, depending on the access to surface quantities, algorithms that are ideal in one case may be infeasible in others.  Also, when refining a mesh, element nodes are not necessarily required to lie on the approximated surface, or even alter the approximating surface between iterations.  Continually improving the surface approximation has desirable features, but it also complicates the analysis of convergence and optimality.  Along these lines, developing a method similar to \cite{DD07} where the nodes of the mesh are require to lie on $M$, and thus the surface approximation continually improves, would be of interest.

As was the case in \cite{HMS}, in this paper we have focused on the error $\| \sigma - \sigma_h \|$ for the Hodge Laplacian in the specific case $k=m$.  The results in \cite{HMS}, with the exception of the stability results, applied to general $\mathfrak{B}$ problems and such a generalizing the theory to this class of problems would be a desirable results.  The proofs above introduce no additional complications in generalizing methods for $\mathfrak{B}$ problems to surfaces, and therefore an extension of the results in \cite{HMS} would likely generalize to surfaces.

Analysis of adaptivity in the natural norm, 
$$\| u - u_h \|_{H\Omega^k(M)} + \|\sigma - \sigma_h \|_{H\Omega^{k-1}(M)}  + \|p -p_h \|,$$
is another direction of interest.  Such indicators on polygonal domains are analyzed in ~\cite{DH}, and using the results from \cite{HoSt10a,HoSt10b}, these results can be extended to surfaces with additive geometrical terms.  Analysis of algorithms using these indicators is another area of interest.
